\documentclass[12pt,reqno]{amsart}
\usepackage{amsmath}
\usepackage{amssymb, latexsym, amsfonts, amscd, amsthm, mathrsfs, enumerate, esint}
\usepackage[usenames,dvipsnames]{color}
\usepackage[all]{xy}
\usepackage{graphicx}
\usepackage{epsfig}
\usepackage{hyperref}

\numberwithin{equation}{section}

\definecolor{purple}{rgb}{0.9,0,0.8}

\definecolor{gray}{rgb}{0.7,0.7,0.7}

\newcommand{\abbr}[1]{{\sc\lowercase{#1}}}

\newtheorem{thm}{Theorem}[section]

\newtheorem{lem}[thm]{Lemma}
\newtheorem{ppn}[thm]{Proposition}

\newtheorem{cnj}[thm]{Conjecture}
\theoremstyle{definition}
\newtheorem{defn}[thm]{Definition}

\newtheorem{remark}[thm]{Remark}
\newcommand{\beq}{\begin{equation}}
\newcommand{\eeq}{\end{equation}}

\newcommand{\ep}{\epsilon}

\newcommand{\bA}{\mathbb{A}}
\newcommand{\B}{\mathbb{B}}
\newcommand{\bB}{\mathbb{B}}

\newcommand{\D}{\mathbb{D}}
\newcommand{\bD}{\mathbb{D}}
\newcommand{\E}{\mathbb{E}}
\newcommand{\bE}{\mathbb{E}}

\newcommand{\bG}{\mathbb{G}}
\newcommand{\G}{\mathbb{G}}

\newcommand{\bI}{\mathbb{I}}
\newcommand{\J}{\mathbb{J}}

\newcommand{\bK}{\mathbb{K}}

	\renewcommand{\L}{\mathbb{L}}

	\renewcommand{\P}{\mathbb{P}}	
\newcommand{\bP}{\mathbb{P}}

\newcommand{\Z}{\mathbb{Z}}
\newcommand{\bZ}{\mathbb{Z}}

\renewcommand{\emptyset}{\varnothing}

\renewcommand{\setminus}{\backslash}

\begin{document}

\title[Monotone Interaction of Walk and Graph: Recurrence versus Transience ]
{Monotone Interaction of Walk and Graph:  \\Recurrence versus Transience}
\date{\today}

\author[A.\ Dembo]{Amir Dembo$^*$}
\author[R.\ Huang]{Ruojun Huang$^\diamond$}
\author[V.\ Sidoravicius]{Vladas Sidoravicius$^\dagger$}
\address{$^*$Department of Mathematics, Stanford University, Building 380, 
Sloan Hall, Stanford, CA 94305, USA}
\address{$^{*\diamond}$Department of Statistics, Stanford University,
 Sequoia Hall, 390 Serra Mall, Stanford, CA 94305, USA}
\address{$^\dagger$IMPA, Estrada Dona Castorina 110, Jardim Botanico, 
Cep 22460-320, Rio de Janeiro, RJ, Brazil}

\thanks{This research was supported in part by NSF 
grant DMS-1106627,
by Brazilian CNPq grants 308787/2011-0 and  476756/2012-0, 
Faperj grant E-26/102.878/2012-BBP
and by ESF RGLIS Excellence Network.}



\maketitle
\begin{abstract} We consider recurrence versus 
transience for models of random walks on 
domains of $\mathbb{Z}^d$, in which monotone 
interaction enforces domain growth as a result 
of visits by the walk (or probes it sent), 
to the neighborhood of domain boundary.
\end{abstract}

\section{Introduction}
There has been much interest in studies of 
random walks in random environment 
(see \cite{HMZ}). Of particular 
challenge are problems in which the walker 
affects its environment, as in reinforced random walks. In this context even the most fundamental question of recurrence versus transience is often open. 
For example, the recurrence of two dimensional 
linearly reinforced random walk with large enough reinforcement strength has been recently solved 
in \cite{ACK}, \cite{ST}. The corresponding question 
by M. Keane for once reinforced random walk remains 
open. Moving to $\mathbb{Z}^d$, $d\ge3$, the 
recurrence of once reinforced random walk 
is conjectured to be sensitive to the strength 
of the reinforcement (see e.g. \cite{K1}).  
We consider here certain time-varying, highly 
non-reversible evolutions. 
Specifically, similarly to \cite{DHS}
we study discrete 
time simple random walk (\abbr{srw}) $\{ X_t\}$  
on connected graphs 
$\bG_t \uparrow \bG_\infty \subseteq \overline{\bG}$ (for some 
given, locally finite, connected graph $\overline{\bG}$,
adopting the notation $\bD_t$ in case 
$\overline{\bG} = \mathbb{Z}^d$).
%for some $d \ge 2$). 
That is, starting at given $\bG_0$ and initial site
$X_0 \in \bG_0$, the sequence $\{X_t\}$ is adapted to
some filtration $\{\mathcal{F}_t\}$ with 
$\{\bG_t\}$ being $\mathcal{F}_\cdot$-previsible
(most often using $\mathcal{F}_t=\sigma(X_s, s \le t)$
the canonical filtration of the \abbr{srw}), and having 
$X_t = x$, one chooses $X_{t+1}$ uniformly among all 
neighbors of $x$ within $\bG_t$.
%(and without loss of generality, the movement of walk from $X_{t-1}$ to $X_t$ precedes the change of the graph).
%(Alternatively, one can also consider vertex set %$\mathcal{V}(\bG_t)\equiv\mathcal{V}(\overline{\bG})$ 
%as fixed, only edge set $\mathcal{E}(\bG_t)\uparrow\mathcal{E}(\bG_\infty)\subseteq\mathcal{E}(\overline{\bG})$, 
%as pointed out in related work \cite{ABGK}.) 

Our companion paper \cite{DHS} deals with 
$\{\bG_t\}$ growing independently of $\{X_t\}$, 
a situation in which universality is to be 
expected (c.f. Conjectures 
1.1-1.3 of \cite{DHS} and the analogous 
conjectures made in \cite{ABGK} for the 
corresponding strictly positive and finite 
conductances model). In contrast, rich and often
counter intuitive behavior occurs when focusing 
on genuine monotone interaction between 
the path $\{X_0,\ldots,X_t\}$ of the walk and 
the growth $\bG_{t+1} \setminus \bG_t$ of the
graphs. In this context, our
%for \abbr{srw}-transient $\overline{\bG}$
Lemma \ref{thm0} provides an 
equivalent condition for transience/reccurence 
of the \abbr{srw} on growing $\{\G_t\}$. We
examine here its consequences for various 
monotone interactions. In particular, in 
{\emph{open by touch}} type interactions
a local or bounded number of edges 
is added to $\bG_t$ as result of each visit 
to its boundary sites (see Defn. \ref{def:obt}).
We then expect the \abbr{srw} on $\G_t$ to inherit 
the transience of $\overline{\G}$ when
starting at large enough $\G_0$ (see
Prop. \ref{obstacles} and Conjecture 
\ref{cnj:perc-sup-crit}), 
while it should be recurrent when $\G_0$ is 
small and $\G_t$ almost regular (see 
Defn. \ref{def:ars} and Prop. \ref{conditional},
but beware the counter example of transience
provided in Prop. \ref{biased_opening}). 
This recurrence should be related to such interaction 
requiring order of surface-area visits to the 
graph's boundary. In the same direction we find
sharp transition between transience and recurrence 
at lower than 
surface-area growth of the number of boundary visits
for {\emph{expanding glassy spheres}} 
type interactions. These are 
of almost regular shape due to global growth 
upon completion of the required number of visits
to the current graph's boundary, see Defn. \ref{def:egs}
and Prop. \ref{glassy}. 
Finally, we consider 
{\emph{probing simple random walk}} where  
a variable/fixed number of guided/unguided probes 
is sent from walker's current location, with 
each probe adding a site at the graph's boundary. 
In this setting 
guided probes may flip the walk between 
transience and recurrence, whereas for 
unguided probes the \abbr{srw} supposedly 
inherits the transience/recurrence of the underlying 
graph $\mathbb{Z}^d$ (see Prop. \ref{general_soldier} and 
Conjecture \ref{cnj:fixed-unguided-probes}). 

\medskip
Recall that for any time-homogeneous Markov chain $\{Z_t\}$ on countable state space $\overline{\bG}$, 
a zero-one law applies for the recurrence of state $z \in \overline{\bG}$, namely for the event 
$\{N_z = \infty \}$ and 
$N_z:=\sum_t \bI_{\{z\}} (Z_t)$. Further, such 
recurrence, i.e. $\bP_z(N_z=\infty)=1$, is 
equivalent to $\bE_z(N_z)=\infty$ and to 
$\bP_z(Z_t=z$ for some $t)=1$. In contrast,
neither such equivalence, nor zero-one law apply in
our more general setting of monotonically interacting 
\abbr{srw} on growing graphs. For example, both zero-one 
law and equivalence break for suitable choices of $\{1,\ldots,\infty\}$-valued random variable $K$, taking $\bG_t$, $t \le K$ 
a single edge adjacent to the origin and 
$\bG_t=\bZ^3$ for $t>K$.
This prompts our selection hereafter of the following definition of sample-path recurrence.
%As zero-one law does not hold in general in our setting, %we distinguish between different notions of recurrence %or transience that were equivalent in classical theory, %specifically, in this note we have
%
\begin{defn}\label{def:rec}
A site $x \in \overline{\bG}$ is recurrent for the
sample path of \abbr{srw} $\{X_t\}$ on $\{\bG_t\}$ 
if $\{X_t = x$ i.o$\}$. Otherwise we say that 
the site $x \in \overline{\bG}$ is transient for 
this sample path of the \abbr{srw} on $\{\bG_t\}$. 
\end{defn}

Let $\text{deg}_{\bG}(z)$
denote the degree of vertex $z$ in graph $\bG$, 
$d^{\bG}(x,y)$ the graph distance in $\bG$ between $x,y\in \bG$ and $\bB^{\bG}(z,r)$
the corresponding (closed) ball of radius $r$ and
center $z$ in $(\bG,d^{\bG})$, with $\bB_r$ 
denoting the projection on $\bZ^d$ of the (closed) 
Euclidean ball of radius $r$ centered at 
the origin. Hereafter, we set $X_0=0$, assuming   
that $\bB^{\overline{\bG}}(0,1) \subset \bG_0$,
and in view of the following lemma,  
focus without loss of generality on 
sample path recurrence of this distinguished vertex. 

\begin{lem}\label{lem:irred}
For any \abbr{srw} $\{X_t\}$ on monotone increasing 
connected graphs $\{\bG_t\}$, on the event 
$\{X_t=0,\text{ f.o.}\}$ of transience of $0$
we have that a.s. $d^{\bG_t}(0,X_t) \to \infty$ 
when $t \to \infty$. Consequently, with probability one,
$0$ is transient for the sample path of the \abbr{srw}
if and only if every vertex of $\overline{\bG}$ is 
transient for this path.  
% \red{(Q: need this vertex to be in $\bG_\infty$. But when $\{\bG_t\}$ is random, how to formulate it?)}
\end{lem}
\begin{proof} Fixing 
%$t \ge 0$ and 
$r$ finite, let 
$A_t=\{X_{t+u}=0$, some $u \ge 0\}$ and
$\Gamma_{t,r}=\{d^{\bG_t}(0,X_t) \le r\}$. Since
$X_0=0$, and $\bG_t \subseteq \overline{\bG}$ 
are non-decreasing (so in particular
$\bB^{\bG_t}(0,r)\subseteq \bB^{\overline{\bG}}(0,r)$, 
for any $t \ge 0$), and 
$\overline{\bG}$ is locally finite, 
it follows that
$$
\bE_{X_t}(\bI_{A_t}|\mathcal{F}_t) \ge M_r^{-r} 
\mathbb{I}_{\Gamma_{t,r}} \,,
$$
with 
$M_r:=\max_{z\in\bB^{\overline{\bG}}(0,r)}\{\text{deg}_{\overline{\bG}}(z)\}$ finite. 
%Suppose for contradiction that on $\{X_t=0,\text{ f.o.}\}$, we have with positive probability 
%$\liminf_{t\rightarrow\infty}d^{\bG_t}(0,X_t)=r<\infty$; 
%that 
%is, 
%$X_{t_k}\in\bB^{\bG_{t_k}}(0,r)$, for some %$\{t_k\}\uparrow\infty$. 
%Since graphs $\bG_t$ are non-decreasing,
%\begin{align*}
%\inf_{z\in\bB^{\bG_{t_k}}(0,r)}\bP_z(\text{\abbr{SRW} }\{X_{t_k+s}\}_{s\ge0}\text{ on }\{\bG_{t_k+s}\}_{s\ge0}\text{ ever hits }0)\ge(1/M_r)^r>0,
%\end{align*}
When $t \to \infty$ we have that 
$\mathbb{I}_{A_t} \to \mathbb{I}_{\{X_t=0\;\; i.o.\}}$
and 
$$
\liminf_{t \to \infty} \mathbb{I}_{\Gamma_{t,r}} 
= \mathbb{I}_{\{X_t \in \bB^{\bG_t}(0,r) \; i.o.\}}
\,,
$$
hence by L\'evy's upward theorem, w.p.1. 
if $\{X_t \in \bB^{\bG_t}(0,r)$ i.o.$\}$ then 
$\{X_t=0\text{ i.o.}\}$.
% (see e.g. \cite[Section 4]{DHS} or \cite[Theorem 6.3.3]{Du}). 
Taking $r\to \infty$ we deduce that w.p.1. transience at $0$ of the sample path implies finitely many visits of 
$X_t$ to $\bB^{\bG_t}(0,r)$ for each $r$, hence both transience
of every $x \in \overline{\bG}$ for this sample path 
and that $d^{\bG_t}(0,X_t) \to \infty$ when $t \to \infty$. 
%Consequently, $0$ is transient w.p.1 implies a.s. %$d^{\bG_t}(0,X_t) \to \infty$ when $t \to \infty$, thereby any 
%$z\in\bG_\infty$ is transient ($z\in\overline{\bG}\backslash\bG_\infty$ is never visited), since $d^{\bG_t}(0,z)\le d^{\bG_{\kappa_z}}(0,z)<\infty$, for all $t>\kappa_z:=\inf\{t: X_t=z\}$. 
%To show the converse, given that $z\in\overline{\bG}$ is transient w.p.1 , either $\kappa_z<\infty$, which then implies that $0$ is transient, by applying the argument in previous paragraph to \abbr{SRW} $\{X_{\kappa_z+s}\}_{s\ge0}$ on $ \{\bG_{\kappa_z+s}\}_{s\ge0}$ with roles of $0$ and $z$ interchanged; or $\kappa_z=\infty$, which also means $X_t=0\text{ f.o.}$ (since if $X_t=0$ i.o., one can show that $\kappa_z<\infty$, for all $z\in\bG_\infty$ by employing L\'evy upward theorem once again).
\end{proof}

Next, with $\text{deg}_{\overline{\bG}}(z)$ the 
maximal possible degree of vertex $z$, 
we define the boundary set 
$$
\partial\bG_t:=\{z\in\bG_t: \text{deg}_{\bG_t}(z)<\text{deg}_{\overline{\bG}}(z)\}\,,
$$
of $\bG_t$ (consisting of all vertices of 
$\bG_t$ whose degree may yet change as $\bG_t\uparrow\bG_\infty$), 
and characterize transience via summability of 
$p_n:=\mathbb{P}(A_n|\mathcal{F}_{\eta_n})$,
for events 
$$
A_n:=\{\exists s\in[\eta_n,\sigma_n): X_s=0\}\,,
$$
and the following 
%monotone seqeunces of 
$\mathcal{F}_t$-stopping times 
$\{\eta_n,\sigma_n\}$, starting at $\eta_0=0$:
\begin{align*}
\sigma_n:=&\inf\{t\ge\eta_n: X_t\in\partial\bG_{\eta_n}\},\quad n\ge0\\
\eta_{n+1}:=&\inf\{t\ge\sigma_n: X_t\not\in\partial\bG_t\}.
\end{align*}
\begin{lem} \label{thm0}
Let $S:= \sum_n p_n$.
\newline
(a) The sample path of $\{X_t\}$ is a.s. recurrent
on $S=\infty$;\\
(b) Conversely, if \abbr{srw} on the fixed graph $\overline{\bG}$ is transient, then 
the sample path of $\{X_t\}$ is a.s. transient
on $S<\infty$.
\end{lem}

\begin{remark}\label{rmk:thm0}
In \cite[Sections 4,5]{ABGK} it is shown 
that a monotonically interacting strictly positive and 
finite conductance model on a tree $\overline{\bG}$ 
tends to follow the recurrence/transience of its 
starting and ending conductances (in particular, this 
applies for $\overline{\bG}=\bZ$). However, this 
approach, based on using flows to construct suitable 
sub or super martingales, is limited in scope to trees
(indeed \cite[Section 6]{ABGK} provides a counter example
to such conclusion in case $\overline{\bG}=\bZ^2$). 
In contrast, while less explicit, Lemma \ref{thm0} 
applies for any $\overline{\bG}$. Further, the 
advantage of this lemma lies in $p_n$ being the probability that 
a \abbr{SRW} on fixed graph $\overline{\bG}$  
%$\bG_{\eta_n}\backslash\partial\bG_{\eta_n}$ (and in turn equivalently on fixed graph $\overline{\bG}$) 
starting at the random position $X_{\eta_n}$ 
visits $0$ before $\partial\bG_{\eta_n}$, 
hence amenable to the use of classical hitting 
probability estimates for random walk on a fixed graph.
\end{remark}

\begin{proof} Recall Paul L\'evy's extension of 
Borel-Cantelli lemma (see \cite[Theorem 5.3.2]{Du}), 
that a.s. $S=\infty$ if and only if $\{A_n$ i.o.$\}$
which immediately yields part (a). Further, 
$\text{deg}_{\bG_t}(0) = \text{deg}_{\overline{\bG}}(0)$
for all $t$ (by our assumption that 
$\bB^{\overline{\bG}}(0,1) \subset \bG_0$), hence 
$X_s \ne 0$ whenever $s \in [\sigma_n,\eta_n)$ and 
the a.s. transience of $0$ for $\{X_t\}$ 
in case $S<\infty$ follows, provided $\sigma_n<\infty$ for all $n$.
To rule out having with positive probability 
$\{\sigma_n=\infty$ and $X_t=0$ for 
infinitely many $t \ge \eta_n\}$, note that by our assumption 
of transience of the \abbr{srw} on $\overline{\bG}$, 
the former can not occur if $\partial\bG_{\eta_n} = \emptyset$.
So, assuming hereafter that $\partial\bG_{\eta_n}$ is 
non-empty, conditional on $\mathcal{F}_{\eta_n}$, 
if the irreducible \abbr{srw} on the fixed connected 
graph $\bG_{\eta_n}$ visits $0$ i.o., then it a.s. would   
also enter $\partial\bG_{\eta_n}$ in finite time, namely 
having $\sigma_n<\infty$.  
\end{proof}

Of particular interest to us are the 
{\emph{open by touch}} type interaction models,
in which graph growth occurs only upon 
the walker's visits of the graph's boundary sites.
\begin{defn}\label{def:obt}
We say that $Y_t \in \bG_t\uparrow\bG_\infty 
\subseteq \overline{\bG}$ is an open by touch (\abbr{obt}) 
interaction model, if $\bG_{t+1}=\bG_t$
except when $Y_t\in\partial \bG_t$, at which times 
all edges of
$\bB^{\overline{\bG}}(Y_t,1)$ are added to $\bG_{t+1}$.
More generally, in a {\emph{partially open by touch}}
(\abbr{pobt}) interaction we add to $\bG_{t+1}$, 
with uniformly bounded away from zero probability,
one (or more) of the edges adjacent to 
$Y_t \in \partial \bG_t$, in a {\emph{first open by touch}}
(\abbr{fobt}) interaction such addition of edges 
adjacent to $x \in \partial \bG_t$ occurs 
{\emph{only at the first}} visit of $x$ by $Y_t$,
whereas in a 
{\emph{remotely open by touch}} (\abbr{robt}) 
we only require that the collection 
of edges $\bA_t$ added to $\bG_t$
when $Y_t \in \partial \bG_t$ be of uniformly 
bounded cardinality.
\end{defn} 

In any \abbr{obt} model the walker opens 
with a one step delay all edges of $\overline{\bG}$ adjacent to 
its current position, in effect performing \abbr{srw} on $\overline{\bG}$, 
except at her first visit to certain sites. Thus, one may expect 
that \abbr{obt} interactions inherit the transience of \abbr{srw} 
on $\overline{\bG}$, as long as they follow that \abbr{srw} 
update rule, except maybe when at $\partial \bG_t$. Indeed, 
we utilize such notion of {\emph{extended simple random walks}} 
when studying this question (in the sequel).
\begin{defn}\label{defn:esrw}
We say that a \abbr{robt} interaction model
on $\bG_t \uparrow \bG_\infty \subseteq
\overline{\bG}$ forms an extended simple random walk 
if $Y_t$ follows the steps of \abbr{srw} on 
$\overline{\bG}$ except for allowing 
whenever $Y_t \in \partial \bG_t$ to have 
any $\mathcal{F}_t$-measurable mechanism for 
choosing $Y_{t+1} \in \bG_t \cap \mathcal{C}(Y_t)$,
for some fixed 
$\mathcal{C}(x):=\bB^{\overline{\bG}}(x,r(x))$,
$c<1$ and $1 \le r(x) \le c d^{\overline{\bG}}(x,0)$.  
\end{defn}

Note however that in this context, \cite{K2}
shows that for transient $\overline{\bG}=\mathbb{Z}^d$, 
$d \ge 3$, starting at $Y_0=0$ and $\bG_0=\bB(0,1)$,
it is possible to have an 
\abbr{obt} extended simple random walk,
whose sample path is a.s. recurrent at $0$,
even with $r(x)=1$ everywhere. Specifically, this is 
done by creating drift to the origin at first 
visits, such that 
$\bE[Y_{t+1}-Y_t|\mathcal{F}_t] = -\delta Y_t/\|Y_t\|_1$
for some $\delta>0$ and all $Y_t \in \partial \bG_t$.
In contrast, with Lemma \ref{thm0}  
applicable for extended simple random walks, 
we next prove that $0$ is a.s. transient for the 
sample path in any \abbr{pobt} starting 
with $\bD_0 \subseteq \mathbb{Z}^d$, $d\ge3$ 
of fast enough diminishing density of closed edges. 
\begin{ppn} \label{obstacles}
If \abbr{srw} on $\overline{\bG}$ is transient, all 
sites (but not all edges), of $\overline{\bG}$ are in $\bG_0$ and  
\begin{align}\label{obt-diminish}
S_\star := \sum_{x\in\partial\bG_0}\sup_{y\in\mathcal{C}(x)} \{ \mathbb{P}_y(\text{\abbr{SRW} on }\overline{\bG} \text{ ever hits }0) \} <\infty, 
\end{align}
then almost every sample path in any \abbr{obt} 
extended simple random walk $Y_t$ 
on $\bG_t\uparrow\bG_\infty$ is transient. The 
same applies for any \abbr{pobt} provided   
$\overline{\bG}$ is of uniformly bounded degrees.
\newline
In particular, in case $\overline{\bG} = \mathbb{Z}^d$, $d\ge3$, 
denote by $N(k)$ the number of vertices in $\partial\bD_0$ that 
are on the boundary of the box of side length $k$ 
centered at $0$. Then w.p.1. the sample path of \abbr{pobt} extended simple random walk 
$Y_t$ on $\bD_t\uparrow\bD_\infty$ is transient, provided 
\begin{align}\label{obt-zd}
\sum_{k=1}^\infty N(k)k^{2-d}<\infty\,.
\end{align}
\end{ppn}

\begin{remark}\label{rmk:conv-obt}
There is no analog of Proposition \ref{obstacles} 
for recurrent $\overline{\bG}$. For example, taking
$\overline{\bG}=\bZ^2$ with $Y_0=0$ in connected 
$\bG_0$ whose closed edges consist of exactly one 
among those touching sites $(\pm s_t,0)$ and
$(0,\pm s_t)$, for $s_t=[(1+c)^{t-1}]$, $t \ge 1$ 
and $c<1$ (hence of fast diminishing density), one 
can create \abbr{obt} extended simple random walk with 
only four sample path, $Y_t= (\pm s_t,0)$
or $Y_t=(0,\pm s_t)$, all of which are transient. 
\end{remark}
\begin{remark}
Proposition \ref{obstacles} applies 
regardless of the manner and probability in 
which edges are added to $\partial \bG_t$ in 
the \abbr{pobt} interaction, but this may be a 
somewhat delicate matter when starting with 
smaller graph $\bG_0$. For example,
%for transient $\overline{\bG}$ of uniformly bounded 
%degrees, to interactions for which whenever 
%$Y_t\in\partial \bG_t$ at least one additional
%edge of $\bB^{\overline{\bG}}(Y_t,1)$ 
%is added to $\bG_{t+1}$  
%with uniformly bounded away from zero probability. 
walking on sub-domains $\bD_t$ of
the recurrent $\overline{\bG}=\bZ^2$, starting 
at $\bD_0=\bB (0,1)$, Proposition \ref{biased_opening} 
proves a.s. transience in \abbr{fobt} interaction 
for which only the right/up/down edges out of each site 
are added to $\bD_{t+1}$ (upon first visit to the site
by the \abbr{srw} on $\bD_t$).
%(upon first visit to 
%$x$ by the \abbr{srw} on $\bD_t$, and no changes 
%to $\bD_t$ in subsequent visits to $x$). 
\end{remark}

While Proposition \ref{obstacles} requires diminishing density 
of closed edges in $\bD_0 \subseteq \bZ^d$, $d \ge 3$, we 
believe that \abbr{srw} with \abbr{obt} interaction has 
a.s. transient sample path as soon as the open edges of 
$\bD_0$ percolate in $\bZ^d$. Specifically, we make the 
following conjecture.
\begin{cnj}\label{cnj:perc-sup-crit}
For the \abbr{obt} interaction in $\bZ^d$, $d \ge 3$, 
upon starting its \abbr{srw} at $0 \in \bD_0$, if $\bD_0$ 
is the infinite cluster of bond/site super-critical 
percolation, then the corresponding sample-path is a.s. 
transient.
\end{cnj}

Having seen the effect of the initial graph on recurrence 
versus transience for certain interacting walks and graphs, 
we turn to the implications of asymptotic regularity 
of $\bG_t$. To this end, we first 
consider {\emph{expanding glassy spheres}} interactions, 
in which growth requires certain 
number of visits by the walk to the graph's boundary, 
at which point a global expansion of the graph occurs.
\begin{defn}\label{def:egs}
Fix $c \ge 1$, $N(k) \ge 1$ and infinite 
(connected, locally finite), graph $\overline{\bG}$, 
setting $\overline{\bB}_k:=\bB^{\overline{\bG}}(0,ck)$.
The {\emph{expanding glassy spheres}}
(\abbr{egs}) interaction consists of 
\abbr{srw} $Z_t$ on $\bG_t=\overline{\bB}_k$
for $t\in[\tau_k,\tau_{k+1})$, starting at $Z_0=0$ 
and with $\mathcal{F}^Z$-stopping times $\tau_1:=0$, 
$$
\tau_{k+1}:=\inf\{s > \tau_k : \sum_{t=\tau_k}^{s-1}
\mathbb{I}_{\partial\overline{\bB}_k} (Z_t) =N(k)\}
\,,\; k \ge 1 \,.
$$  
Alternative sets may be used as well. For example, 
as the name \abbr{egs}  suggests, in case 
$\overline{\bG}=\mathbb{Z}^d$, $d \ge 1$, we define 
such \abbr{egs} $\{Z_t\}$ as being confined to 
the projected Euclidean ball $\bB_{ck}$ 
until making the prescribed number of visits 
$N(k)$ to its boundary, at which time this 
projected ball expands to $\bB_{c(k+1)}$, and so on 
(instead of using the graph distance
on $\bZ^d$ for defining such balls).
\end{defn}

Employing Lemma \ref{thm0} we determine 
the transition between recurrence and transience 
for \abbr{egs} on $\bZ^d$, $d\ge 2$ 
in terms of asymptotic growth of the prescribed 
hit counts $\{N(k)\}$ (showing in particular 
that for 
$d=2$, such \abbr{egs} is always recurrent).
\begin{ppn} \label{glassy} 
For $\overline{\bG}$ of bounded degrees
and   
$\mathcal{C}_k:=\{x\in\overline{\bB}_k :
d^{\overline{\bG}}(x,\partial\overline{\bB}_k)=1\}$,
\begin{align}\label{egs-rec}
\sum_{k=1}^\infty N(k) \inf_{x\in\mathcal{C}_k}\mathbb{P}_x(\text{\abbr{srw} on }\overline{\bG}\text{ hits }0\text{ before }\partial\overline{\bB}_{k})=\infty,
\end{align}
yields a.s. sample path recurrence for
the corresponding \abbr{egs}, whereas if
\begin{align}\label{egs-trans}
\sum_{k=1}^\infty N(k)\sup_{x\in\mathcal{C}_k}\mathbb{P}_x(\text{\abbr{srw} on }\overline{\bG}\text{ hits }0\text{ before }
\partial\overline{\bB}_{k+1})<\infty,
\end{align}
then a.s. the sample path of the corresponding 
\abbr{egs} is transient. In particular, 
for $\overline{\bG}=\bZ^d$, $d\ge 2$, 
we have zero-one law for transience/recurrence of 
the \abbr{egs} interaction sample path which 
is a.s. 
%recurrent in case $d=1,2$ and for an
%\abbr{egs} on $\bZ^d$, $d \ge 3$, such 
%sample path is a.s. 
transient if and only if 
\begin{align}\label{egs-zd}
\sum_{k=1}^\infty N(k)k^{1-d}<\infty \,.
\end{align}
\end{ppn}
\begin{remark} The assumed uniform bound on 
deg$_{\overline{\bG}}(z)$ is only required for getting 
the factor $N(k)$ within the sum on \abbr{lhs} of \eqref{egs-rec}.
To this end it suffices to have a uniform (in $k$ and $z$), upper 
bound on the expected hitting time of $\mathcal{C}_k$ by the 
\abbr{srw} on $\overline{\bB}_k$ starting at 
$z \in \partial \overline{\bB}_k$. Some such 
condition is relevant for recurrence/transience of \abbr{egs}.
Indeed, consider $c=1$ and $\overline{\bG}$ arranged in layers, 
with each site in $k$-th layer (i.e. of distance $k$ from $0$), 
having $\ell^{\pm}_{k}=\ell_k q_k p_k^{\pm} \ge 1$ edges to 
$(k \pm 1)$-th layer (with $\ell_0^-=0$), and 
$\ell^0_k = \ell_k (1-q_k) p_k^-$ edges to its own layer. 
For \abbr{egs} on such graph, $t \mapsto 
d^{\overline{\bG}}(0,Z_t)$ evolves up to holding 
times within layers, as a modified birth-death 
chain $W_s$ on $\bZ_+$, starting at $W_0=0$, moving 
with probability $p^\pm_k$ from $k$ to $(k \pm 1)$, 
but opening edge $(k,k+1)$ only after 
an independent Binomial$(N(k),q_k)$ steps from 
$k$ to $(k-1)$ are made by $\{W_s\}$. Conditions 
\eqref{egs-rec} and \eqref{egs-trans} amount 
to divergence and convergence, respectively, of 
$\sum_k N(k) \bP_{k-1}(W_s$ hits $0$ before $k)$,
which for uniformly bounded $\{q_k\}$ is indeed 
a sharp criterion for recurrence/transience 
of $\{W_s\}$ (and thereby of the \abbr{egs}). 
However, the latter series is missing the factor 
$q_k$, so for unbounded $\{q_k\}$ it is often 
wrong for determining transience versus recurrence
of $\{W_s\}$.  
\end{remark}

Building on the insight provided by Proposition \ref{glassy},
%Adapting the proof of Lemma \ref{thm0} 
we relate the regularity, as defined below, of the 
graphs $\bD_t$ produced by an \abbr{robt} 
interaction model on $\mathbb{Z}^d$, $d \ge 2$, with 
the a.s. 
sample path recurrence for the corresponding \abbr{srw}.  
\begin{defn}\label{def:ars}
We say that $\bK$ is a $\gamma$-{\emph{almost regular shape}}
for growing domains $\{\bD_t\}$ in $\bZ^d$, $d \ge 2$, 
if $\bD_t \supseteq f(t) \bK \cap \bZ^d$ for all $t$ large 
and some non-decreasing $f(\cdot) \ge 1$, such that 
$d^{\bD_t}(z, f(t) \bK) \le \gamma \log f(t)$ for all $z \in \bD_t$.
\end{defn}
\begin{ppn} \label{conditional}
%\label{indep_growing}
There exist $c_d>0$, such that if $0 \in \bD_0$ a 
finite connected domain in $\mathbb{Z}^d$, $d\ge2$ 
and the ball $\bB$ is $c_d$-almost regular shape 
for the growing domains $\bD_t$, then $0$ is a.s. 
recurrent for the sample path of
%\newline
any \abbr{srw} $\{R_t\}$ of \abbr{robt} interaction with $\bD_t$.
%\newline
%(b) For such domains $\bD_t$ in $\bZ^2$ that grow independently of 
%the \abbr{srw} $\{M_t\}$ on them, provided either
%$t^{-1/2} f(t)$ is bounded above or converges to $\infty$ as $t$ grows.
\end{ppn}
\begin{remark} 
If the range of an \abbr{robt} interaction model 
$\{R_s, s \le t\}$ contains the whole $f(t) \bB$ 
ball within $\bD_t$, one may be tempted to 
conclude that recurrence of $\{R_t\}$ then 
trivially follows since every site of $\bZ^d$ 
is for sure being visited at least once. 
However, as mentioned before, unlike \abbr{srw} 
on fixed graph, in case of growing domains the 
walk may nevertheless w.p.1 return to $0$ only 
finitely many times. So,
Proposition \ref{conditional} provides a non-trivial 
conclusion, even in this setting. 
\end{remark}

%\begin{remark}
%Independently growing domains, as in part (b) of Proposition \ref{conditional}, are considered in \cite{DHS}, where phase
%transition between recurrence and transience is proved for 
%regular domains $\bD_t = f(t) \bK \cap \bZ^d$, $d \ge 3$. 
%It is further conjectured that \abbr{srw} on any 
%independently growing domains in $\bZ^2$ is recurrent 
%(see \cite{DHS} and references therein). However, the approach
%of \cite{DHS} relies on invariance principle to relate 
%this with recurrence of reflected Brownian motion on growing 
%domains, hence likely applicable only for regular domains.
%\end{remark}
\begin{remark}
For any $\bK \subset \mathbb{R}^d$ and $f>0$, let
$(f\bK)_d = f \bK \cap \bZ^d$ denote the corresponding 
lattice projection. Proposition \ref{conditional} 
then holds for any $\bK$ such that for some $c$ finite,
\begin{align}\label{ars-K}
\liminf_{f\to\infty} \;
\inf_{\{z\in (f \bK)_d : d(z,\partial (f\bK)_d ) \ge c\log f\}} \;
\frac{1}{\log f}\log
\bP_z(\text{\abbr{SRW} hits }0\text{ before }\partial(f\bK)_d)
>-d\,.  
\end{align}
\end{remark}

%Moving to the next example, 
%similar argument also works when $\bD_t\uparrow\mathbb{Z}^2$ is independent of the \abbr{SRW} $\{M_t\}$ (as opposed to interacting), and $\bD_t$ grows like a ball of logarithmic fluctuations in graph distance. 
%\begin{defn}
%In $\mathbb{Z}^2$, given deterministic $\{\bD_t\}$ such that $\bD_t\supseteq\bB_{f(t)}$ for $t$ large enough, where $f(t)$ increasing satisfies $\limsup_{t\rightarrow\infty}\frac{f(t)}{t^{1/2}}<\infty$, and $d^{\bD_t}(z,\bB_{f(t)})\le c_0\log f(t)$ for all $z\in\bD_t$.
%\end{defn}
%\begin{ppn} 
%There exists $c_0>0$ such that \abbr{SRW} $\{M_t\}$ on $\{\bD_t\}$ is recurrent almost surely.
%\end{ppn}
%\begin{remark}
%As \abbr{SRW} in $\mathbb{Z}^2$ is diffusive, we have that whenever $\liminf_{t\rightarrow\infty}\frac{f(t)}{t^{1/2}}=\infty$, $\{\bD_t\}$ grows faster than the typical speed of \abbr{SRW}, hence $\{M_t\}$ is trivially recurrent.
%\end{remark}

Transience w.p.1. is proved in \cite[Sect. 6]{ABGK} for
the following monotone increasing conductance model 
on edges of $\bZ^2$: starting at $t=0$ with walker at 
the origin and conductance $1$ at each edge, upon 
walker's first visit of each vertex, the conductances 
of its adjacent edges to the right/up/down are increased to $2$.  
Adapting the arguments of \cite[Sect. 6]{ABGK}, we next provide
examples of a.s. transient \abbr{fobt} interaction between 
\abbr{srw} $\{E_t\}$ and the corresponding growing domains 
$\{\bD_t\}$ in $\bZ^2$, emphasizing 
%in particular 
the role of the initial graph $\bD_0$.  
%\begin{defn}\label{def-bias-open}
\begin{ppn} \label{biased_opening}
Consider the \abbr{srw} $\{E_t\}$ on $\bD_t \subseteq \bZ^2$, 
that starts from $E_0=0 \in \bD_0$ and opens only the 
three right/up/down edges adjacent to each site that it first 
visits (where after each such opening the walk stays 
put for one step before choosing its next position, 
now on $\bD_{t+1}$).  
\newline
(a) If $\bD_0$ consists of the vertices of $\bZ^2$ with each edge 
of $\bZ^2$ independently chosen to be in $\bD_0$ with same 
probability $p \in [0,1)$, then the sample path of $E_t$ is 
$\bP_p$-a.s. transient. 
\newline
(b) Alternatively, the sample path of $\{E_t\}$ is a.s. transient
whenever $k^{-r} |\bD_0\cap[-k,k]^2| \to 0$ as $k \to \infty$, 
for some constant $r<3/4$.
\end{ppn}

Our final result deals with a.s. transience for 
the {\emph{probing simple random walk}} (\abbr{psrw}),
$\{K_t\}$ on growing domains $\{\bD_t\}$ in $\bZ^d$, 
$d \ge 2$. Starting at $K_0=0$ and $\bD_0=\{0\}$, such 
\abbr{psrw} is allowed to send at time $t$ some  
$\mathcal{F}_t$-adapted number of probes $m(t)$, with
each probe adding precisely one site to $\bD_t$
(and opening all relevant edges connecting those sites with 
the existing graph), prior to the walk's move from 
$K_t$ to $K_{t+1} \in \bD_{t+1}$. The aim of 
the \abbr{psrw} is to guarantee a.s. transience 
of its sample path with minimal asymptotic 
running average number of probes
$\overline{m}_t:=t^{-1} \sum_{s=1}^t m(s)$. Conversely, 
the \abbr{psrw} may aim at a.s. recurrence of 
its sample path with a maximal 
asymptotic running average 
number of probes.
In different
versions of this problem the \abbr{psrw} may or may
not have control on the probes locations and the
number of probes being used in each step. 
\begin{ppn}\label{general_soldier}
$~ $
\newline
(a) For $\bZ^d$, $d\ge2$ and any $\ep>0$, there exist 
$\mathcal{F}_t$-adapted $\{m(t)\}$ and choices of 
the $m(t)$ probe positions at $\bZ^d$-distance-one 
from $\bD_t$, such that eventually $\overline{m}_t<\ep$ and 
the sample path of $K_t$ is a.s. transient. There also exist (some other) such probe numbers and locations for which eventually $\overline{m}_t > \ep^{-1}$ and the 
sample path of $K_t$ is a.s. recurrent. 
\newline
(b) Suppose each probed site is chosen according to the hitting
measure of $\bD_t^c$ by a \abbr{srw} on $\bZ^d$ which starts at
the current position $K_t$  of the \abbr{psrw}. Then there exist 
finite constants $c_d$ and $\mathcal{F}_t$-adapted process $\{m(t)\}$ such that a.s. 
$\limsup_t \overline{m}_t < c_d$, the \abbr{psrw} sample path 
is transient in case $d \ge 3$, and recurrent with 
$\liminf_t \overline{m}_t$ arbitrarily large, in case $d=2$. 
\end{ppn}

\begin{remark}
Two obvious open problems are whether 
part (b) of Proposition \ref{general_soldier}  
holds for any $c_d>0$, $d \ge 3$, and whether
in this context one can also select a process
$m(t)$ yielding a.s. sample path recurrence 
when $d \ge 3$ and transience when $d=2$. 
\end{remark}
We end with the following conjecture and 
related open problems.
\begin{cnj}\label{cnj:fixed-unguided-probes}
In the setting of part (b) of Proposition \ref{general_soldier} 
there exist 
%$\lambda_d < \infty$, an
$\mathcal{F}_t$-adapted $m(t)$ which is uniformly bounded 
above by non-random integer $\lambda_d$, and a \abbr{psrw} having 
a.s. transient sample path when $d \ge 3$, and
a.s. recurrent sample path 
with $m(t) \ge 1$, when $d=2$. 
\newline
If this conjecture is valid, does it apply for $\lambda_d=1$ and does the same apply even 
for constant $m(t)=\lambda_d$
(i.e. removing all control from the \abbr{psrw})?
\end{cnj}

\section{Proof of Propositions \ref{obstacles}, \ref{glassy} and \ref{conditional}} 
%and \ref{indep_growing}}

\noindent
{\emph{Proof of Proposition \ref{obstacles}.}}  
First consider an \abbr{obt} extended simple random walk. 
Recall Remark \ref{rmk:thm0} that $p_n$ is the probability that 
\abbr{srw} on $\overline{\bG}$ starting at $Y_{\eta_n}$ 
visits $0$ before $\partial\bG_{\eta_n}$, 
and Definition \ref{defn:esrw} that 
$Y_{\eta_n}\in \mathcal{C}(Y_{\eta_n-1})$. Hence, 
setting 
$$
g(x):= \sup_{y\in\mathcal{C}(x)} \{ \mathbb{P}_y(\text{\abbr{srw} on }\overline{\bG} 
\text{ ever hits }0)\}\,,
$$
we have that for for any $n \ge 1$, 
$$
p_n 
%= \mathbb{P}_{Y_{\eta_n}}(\text{\abbr{srw} on }\overline{\bG}\text{ hits }%0\text{ before }\partial\bG_{\eta_n})
\le \mathbb{P}_{Y_{\eta_n}}(\text{\abbr{srw} on }\overline{\bG}\text{ ever hits }0)
\le g(Y_{\eta_n-1})\,.
$$
By assumption, all sites of $\overline{\bG}$ are already in $\bG_0$, and with the \abbr{obt} interaction enforcing that $Y_t \notin \partial \bG_{t+1}$, the distinct sites 
$\{Y_{\eta_n-1}, n \ge 1\}$ are all in $\partial\bG_0$. Hence,
$S=\sum_{n \ge 1} p_n$ is bounded above by the assumed finite term $S_\star$ of \eqref{obt-diminish} and the a.s. sample path 
transience of $\{Y_t\}$ follows by part (b) of Lemma \ref{thm0}. In case of an \abbr{pobt} interaction, 
the same derivation yields the bound
$$
S \le \sum_{x\in\partial\bG_0} L_x g(x) \,,
$$
where $L_x$ denotes the number of visits 
by $\{Y_t\}$ to $x \in \partial \bG_0$, up 
to the possibly infinite stopping
time $\theta_x := \inf\{ t \ge 0 : 
\bB^{\overline{\bG}}(x,1) \subseteq \bG_t \}$.  
The \abbr{pobt} interaction adds at least one 
edge to $\bG_{t+1}$ upon each 
visit to $Y_t \in \partial \bG_t$ with probability
at least $\ep>0$. Hence, $\bE[L_x|\mathcal{F}_0] \le \ep^{-1} 
{\text{deg}}_{\overline{\bG}}(x)$. In particular, 
almost surely, $S$ is finite if 
$$
\bE [S|\mathcal{F}_0] \le \ep^{-1} 
\sum_{x\in \partial\bG_0} 
{\text{deg}}_{\overline{\bG}}(x) g(x) < \infty \,,
$$
which, for $\overline{\bG}$ of uniformly bounded degrees, 
follows from finiteness of $S_\star$.

Specializing to $\overline{\bG}=\bZ^d$, $d\ge3$, 
of uniformly bounded degree, recall Definition 
\ref{defn:esrw} that here 
$\|y\|_1 \ge \|x\|_1 - \|x-y\|_1 \ge (1-c) \|x\|_1$
for any $y \in \mathcal{C}(x)$. Thus, by the elementary 
potential theory formula 
$$
\mathbb{P}_y(\text{\abbr{srw} on }\mathbb{Z}^d\text{ ever hits }0)
\le c_d \|y\|_1^{2-d} \,,
$$ 
for some finite $c_d$ and all $y$ (see \cite[Proposition 1.5.9]{La}),
we get that $g(x) \le \kappa_d \|x\|_1^{2-d}$, for some $\kappa_d
% = c_d (1-c)^{2-d}
$ finite, with condition \eqref{obt-zd} implying that  
$S_\star$ is finite.
\qed\\

\noindent
{\emph{Proof of Proposition \ref{glassy}.}} By 
definition of the \abbr{egs} interaction, necessarily 
$\bG_{\eta_n} = \overline{\bB}_k$ for 
$\eta_n \in [\tau_k,\tau_{k+1})$. To each $k \ge 1$ 
correspond $L_k \in [1,N(k)]$ such stopping times,  
and $Z_{\eta_n} \in \mathcal{C}_k$ for all but 
the smallest of these (namely, $\eta_n=\tau_k$, $k \ge 2$),
in which case $Z_{\eta_n} \in \overline{\bB}_{k-1}$ 
is within distance one of $\partial \overline{\bB}_{k-1}$.
Consequently, $S$ of Lemma \ref{thm0} is bounded 
above by the \abbr{lhs} of \eqref{egs-trans}. Further,
if $\sup_z \text{deg}_{\overline{\bG}}(z) \le \overline{\text deg}$ 
finite, then conditional on $\mathcal{F}^Z_t$,
upon each visit of $\partial \overline{\bB}_k$ 
by $Z_t$ (i.e. time $\sigma_n$), 
we have that $Z_{t+1}$ is not in $\partial \overline{\bB}_k$
with probability at least $\ep := 1/\overline{\text deg}$.
It follows that the collection $\{L_k\}$ stochastically 
dominates the independent Binomial$(N(k),\ep)$ 
variables $\{L'_k\}$, hence $S$ stochastically 
dominates the \abbr{lhs} of \eqref{egs-rec} 
with $N(k)$ replaced there by $L'_k$. The 
latter is the monotone upward limit $T_\infty$ 
of a series $T_n = \sum_{k=1}^n q_k L'_k$, 
with $q_k \in [0,1]$ non-random and 
condition \eqref{egs-rec} amounting to 
$\bE T_n \uparrow \infty$. With var$(T_n) \le \bE T_n$, 
we have that $T_n/\bE T_n \to 1$ in probability, hence 
\eqref{egs-rec} yields that a.s. $S \ge T_\infty = \infty$. 
Our thesis about 
the a.s. transience and recurrence of the corresponding 
\abbr{egs} thus follows from Lemma \ref{thm0} (we note
in passing that the assumption of $\overline{\bG}$ 
transient is only used in part (b) of Lemma \ref{thm0}
for dealing with $\partial \bG_{\eta_n} = \emptyset$, 
which can not occur for \abbr{egs}).

In case of $\overline{\bG}=\mathbb{Z}^d$, $d \ge 3$, $c \ge 1$, 
upon replacing $\overline{\bB}_k$ by $\bB_{ck}$ it 
remains only to verify that our conditions \eqref{egs-rec} 
and \eqref{egs-trans} are equivalent to the divergence, 
respectively convergence, of $\sum_k N(k) k^{1-d}$. This 
follows by potential theory, since  
\begin{equation}\label{eq:unif-pot-bd}
k^{d-1} \mathbb{P}_{x} (\text{\abbr{srw} on }\mathbb{Z}^d\text{ hits }0\text{ before }\partial\bB_{ck}) \,,
\end{equation}
is bounded above and below away from zero, uniformly over 
$k \ge 1$ and $x \in \bB_{ck}$ whose graph distance from 
$\partial \bB_{ck}$ is between $1$ and $2dc$. We note in 
passing that having here $x$ within constant distance 
of $\partial \bB_{ck}$, the standard error term turns 
out to be $O(1)$ (see formula of \cite[Proposition 1.5.10]{La}), 
so for the stated uniform lower bound it must be refined by using 
asymptotics of Green's function (cf. \cite[Page 96]{LL}).
In case $d=2$, the probabilities appearing in \eqref{eq:unif-pot-bd}  
are similarly bounded below by $C/(k \log k)$ for some $C>0$, 
all $k$ and relevant $x$ (see \cite[Propostion 1.6.7]{La}; here 
the error term is refined using the asymptotics of potential kernel, cf. \cite[Page 104]{LL}). With $N(k) \ge 1$, it follows that
in this case \eqref{egs-rec} holds, yielding the a.s. recurrence of 
the \abbr{egs}, in agreement with the divergence of 
$\sum_k N(k) k^{1-d}$ for $d=2$.  
\qed \\

\noindent
{\emph{Proof of Proposition \ref{conditional}.}}
%\newline
%(a) 
With 
%\begin{align*}
$\xi_r:=\inf\{t \ge 0 : \bD_t\cap\bB_r^c\neq\emptyset\}$,
%\end{align*}
denoting the first time the tip of $\bD_t$ reaches the sphere 
of radius $r$ around $0$, 
%For large $m$ 
we construct $L=O(m^d/\log m)$ stopping 
times $\sigma_1<\sigma_2<\cdots<\sigma_L$ within the 
time interval $[\xi_m,\xi_{2m-1}]$,
%in which the tip of $\bD_t$ crosses the annulus 
%$\bB_{2m} \setminus \bB_m$, 
such that for some constant $\delta<1$,
\begin{equation}\label{eq:lbd-cond}
m^{d-1+\delta}
\bP(R_s=0 {\text{ for some }} s \in [\sigma_{\ell},\sigma_{\ell+1})\,|\,\mathcal{F}_{\sigma_{\ell}})
\end{equation}
is bounded away from zero, uniformly in $\ell$ and $m$.
Similarly to the proof of Lemma \ref{thm0}, upon 
considering the union of these events over all 
dyadic $m=2^k$, the a.s. sample path recurrence of the 
\abbr{robt} interacting \abbr{srw} 
%having $\gamma$-almost regular ball shape 
$\{R_t\}$ then follows 
by Paul L\'evy's extension of Borel-Cantelli.
To this end, since 
the Euclidean ball $\bB$ is $\gamma$-almost regular 
for the growing domains $\bD_t$, it follows that 
for all $t$ large and some non-increasing 
$f(\cdot) \ge 1$,
\begin{equation}
\label{eq:ars-fluc}
\bB_{f(t)} \subseteq \bD_t \subseteq \bB_{f(t)+\gamma \log f(t)}\,.
\end{equation}
We set $w :=\gamma \log (2m)$, the maximal 
fluctuation $\gamma \log f(t)$ in shape of $\bD_t$ 
when $t \le \xi_{2m-1}$ (hence $f(t) \le 2m$). 
From \eqref{eq:ars-fluc} 
one has that $\bD_{\xi_{2m-1}} \supseteq \bB_{2m-w}$
(as otherwise $2m - w > f \ge 2m  -\gamma \log f$
for some $f = f(\xi_{2m-1}) \le 2m$, contradictory to 
our choice of $w$). There are thus at least 
$C' m^d$ edges in $\bD_{\xi_{2m-1}} \setminus \bD_{\xi_m}$,
for some universal constant $C'>0$ and all $m$.
Further, domain growth occurs in \abbr{robt} only when 
$R_t \in \partial \bD_t$, and each such boundary visit 
entails adding at most $C'/C$ edges to $\bD_{t+1}$ 
for some universal constant $C>0$ (see Defn. \ref{def:obt}).
So for each $m$ there are at least $C m^d$
such boundary visits within $[\xi_m,\xi_{2m-1}]$.

Hereafter we fix $\ep>0$ small, set $w_\ep=(1+2\ep)w$, $L=C m^d/w_\ep$ 
and consider the stopping times 
$\{\sigma_\ell, 1 \le \ell \le L \}$, with $\sigma_\ell$ 
denoting the $w_\ep \ell$-th smallest $t \ge \xi_m$ such that 
$R_t \in \partial \bD_t$. Turning to prove the stated
uniform probability lower bound for the corresponding 
events per \eqref{eq:lbd-cond}, fix 
$\sigma=\sigma_\ell$ and 
$f=f(\sigma) \in \mathcal{F}_\sigma$. Recall Defn. \ref{def:obt} 
that $d^{\bD_{\sigma}}(R_{\sigma},\bB_f) \le w$, hence
there exists a path in $\bD_{\sigma}$ of length at most 
$w_\ep-1$ leading from $R_\sigma$ to some specific 
$x \in \bB_f$ such that $d^{\bB_f}(x,\partial \bB_f) \ge \ep w$.
Setting $\delta :=(1+2\ep)\gamma\log(2d)$,
the event $\mathcal{A}_\ell$ that the \abbr{srw} 
$\{R_{\sigma+s}, s \ge 0\}$ on $\bD_{\sigma+s}$ 
takes this specific path has probability at least 
$(2d)^{-w_\ep}=(2m)^{-\delta}$. Since 
$\sigma_{\ell+1} \ge \sigma_{\ell} + w_\ep$ 
and $\bB_f \subseteq \bD_{\sigma_\ell}$,
the event considered in \eqref{eq:lbd-cond}
contains the intersection of $\mathcal{A}_\ell$ 
and the event that starting at position $x$ the 
\abbr{srw} on $\bB_f$ 
visits $0$ before reaching $\partial \bB_f$.   
Clearly, $f \ge m - \gamma \log f \ge m - w \ge m/2$ 
(by \eqref{eq:ars-fluc}). Here $x \in \bB_f$ is 
of Euclidean distance at least $C_0 \log m$ from $\partial\bB_f$
for some constant $C_0=C_0(\ep,\gamma,d)>0$, all $m$ and $\ell$.
Hence, by potential theory, the probability that 
\abbr{srw} starting at $x$ visits $0$ before $\partial \bB_f$,
is bounded below by $\kappa_d m^{1-d}$ for some $\kappa_d>0$,
all $d \ge 2$, $m$ and any such $x$ (see
(\ref{eq:unif-pot-bd}) in case $d \ge 3$, 
and text following it for how to handle $d=2$). 
In conclusion, as claimed, uniformly in $\ell$ and $m$, 
$$
\bP(R_s=0 {\text{ for some }} s \in [\sigma_{\ell},\sigma_{\ell+1})\,|\,\mathcal{F}_{\sigma_{\ell}}) \ge 
\kappa_d m^{1-d} \bP(\mathcal{A}_\ell|\mathcal{F}_{\sigma_\ell}) 
\ge \kappa_d m^{1-d} (2m)^{-\delta} \,.
$$\qed

\section{Proof of Propositions \ref{biased_opening} and \ref{general_soldier}}

\noindent
{\emph{Proof of Proposition \ref{biased_opening}.}} 
We call $m \ge 0$ a \abbr{super-non-nv} time if 
$\{(-1,0)+E_m,E_m\}$ is {\emph{unvisited}} by 
$\{E_t, t < m\}$. We further couple our \abbr{fobt} walk 
$\{E_t\}$ to the \abbr{srw} $\{R(t)\}$ on $\mathbb{Z}^2$, 
both starting at $(0,0)$, so that $E_{t+1}-E_t=R({t+1})-R(t)$ 
except if the edge to the left of $E_t$ is not in $\bD_t$, 
in which case with probability $1/4$ both walks have
the same right/up/down increment, while with probability 
$1/12$ each, $E_{t+1}-E_t$ is the right/up/down increment, 
while $R(t+1)-R(t)=(-1,0)$. Clearly,  
$t \mapsto (E_t-R(t))_1$ is then non-decreasing. Moreover,
independently of $\{E_t, R(t), t \le m \}$,
with probability $1/4$ the value of $(E_t-R(t))_1$ 
increases by one at each \abbr{super-non-nv} 
time $m$ for which the edge to the left of $E_m$ 
is not in $\D_0$. Fixing $\ep>0$, let $\mathcal{A}_n$
denote the event that there exist $n^{3/4-2\ep}$ 
\abbr{super-non-nv} times $m \in [0,n]$ with the
edge to the left of $E_m$ not being in $\bD_0$. 
If $\P(\mathcal{A}_n) \ge 1-C n^{-1}$ for some $C$ finite,  
then a.s.
$(E_n-R(n))_1 \ge 0.1 n^{3/4-2\ep}$ for all $n$ 
large (by Borel-Cantelli lemma it 
holds along dyadic $n_k=2^k$, which  
by monotonicity of $(E_n-R(n))_1$
extends to all $n$ large). 
Since $n^{-1/2-\ep} |R(n)| \to 0$, taking $\ep<1/12$ 
yields the stated a.s. sample path transience of $\{E_t\}$.

Adapting \cite[Sect. 6]{ABGK}, we 
proceed to show that indeed 
$\P(\mathcal{A}_n) \ge 1-C n^{-1}$ for some $C$
finite and all $n$. To this end, first 
analogously to \cite[Lemma 6.1]{ABGK}, we know that 
$\inf_n (\log n) \mathbb{P}_{(0,0)}(\mathcal{I}_n) 
\ge C$, for some $C$ positive and events 
\begin{align*}
\mathcal{I}_n:=\bigcap_{t \le n} \Big\{ R(t) \notin 
\{(-1,0), (0,0)\}  \;  \Big\}.
\end{align*}
Next, fixing $n$, 
similarly to \cite[Lemma 6.4]{ABGK} we call 
$m\in[n^{2\ep},n]$ a {\emph{tan time}} 
if 
%the event $E_3(m,n)$ that 
$R[m-\lfloor n^\ep\rfloor,m)$ avoids 
$\{(-1,0),(0,0)\}+R(m)$, and 
$R[0,m-\lfloor n^\ep\rfloor]$ avoids $F+R(m)$, 
%holds, 
for the funnel
\begin{align*}
F:=\{(x,y):x\ge-1, |y|\le\log^3 (n\sqrt{x+2})\}.
\end{align*}
Equipped with this modification of tan time, 
it is easy to adapt the proof of  
\cite[Lemma 6.6]{ABGK}, yielding that 
the number of $n^\ep$-separated tan times 
within $[0,n]$, exceeds $n^{3/4-2\ep}$
with probability at least $1-C n^{-2}$. 
Then, following the proof of \cite[Lemma 6.9]{ABGK},
we deduce that under our coupling, for some $C$ finite, 
with probability at least $1-C n^{-1}$, whenever 
$m<n$ is a tan time for $\{R(t)\}$, at least 
one of $[m-n^\ep,m]$ must be a \abbr{super-non-nv} time 
for $\{E_t\}$. Consequently, with such probability
there are at least $n^{3/4-2\ep}$ \abbr{super-non-nv} 
times $m \in [0,n]$. It thus suffices to show that 
a uniformly bounded away fraction of these times 
has edge left of $E_m$ that is not in $\bD_0$.

\noindent
(a). We reveal whether each of the i.i.d. Bernoulli($p$) 
edges is in $\bD_0$ or not, only when our 
\abbr{fobt} walk $\{E_t\}$ first visits 
one of the two ends of that edge. Hence, ordering 
the \abbr{super-non-nv} times $m_1 < m_2 < \cdots$, 
%of the \abbr{fobt}, for $k=1,2\ldots$, 
since each \abbr{super-non-nv} 
time avoided both the \abbr{fobt} walk current position 
and the lattice site immediately to its left, we have not revealed 
up to time $m_k$ whether the edge left to $E_{m_k}$ is in $\bD_0$ 
or not. Thus, the joint law of events
$\{$edge left to $E_{m_k}$ is not in $\bD_0\}$ stochastically 
dominates the corresponding i.i.d. Bernoulli($1-p$) 
variables. It then
% $\{\widetilde{B}_{m_k}\}$. 
follows that 
with $\mathbb{P}_p$-probability at least $1-C n^{-1}$,
for more than $(1-p)/2$ of the first $n^{3/4-2\ep}$ 
\abbr{super-non-nv} times $m_k$, the edge left 
of $E_{m_k}$ is not in $\bD_0$, as claimed.

\noindent
(b). With $\{E_t, t \le n\} \subset [-n,n]^2$,   
taking $\ep>0$ small enough, our assumption that 
$k^{-r}|\bD_0\cap[-k,k]^2| \to 0$ for $r<3/4-2\ep$
implies that at least half of the edges left to locations
of the \abbr{fobt} walk at the first $n^{3/4-2\ep}$ 
\abbr{super-non-nv} times, are not in $\bD_0$, as claimed. 
\qed 

\medskip
\noindent
{\emph{Proof of Proposition \ref{general_soldier}.}} 
\newline
(a). For any integers $d,L \ge 2$, consider the stretched 
lattice $\L$, consisting of vertices  
\begin{align*}
%\overline{\L} :=
\{ (y_1,...,y_d) \in \Z^d : \text{ at least one }y_i\text{ is integer multiple of } L\},
\end{align*}
and the edges of $\bZ^d$ between them (i.e. connecting pairs of 
vertices from $\L$ whose $\bZ^d$-distance is one).
Denoting by $\J:=(L \Z)^d$ the subset of junction sites in 
$\L$ 
%\begin{align*}
%J:=\{(j_1L,j_2L,...,j_dL)\in\mathbb{Z}^d: \underline{j}\in\bZ^d\}.
%\end{align*}
%The set $S$ together with all edges connecting vertices in $S$ of $\bZ^d$-distance $1$ is called the streched lattice.
and fixing $d \ge 3$, whenever $K_t$ visits a site $z$ of $\L$ for the first time, 
it dispenses one probe per adjacent closed edge of $\L$ (thereby 
using at most $(2d-1)$ probes if $z \in \J$ and at most one probe 
otherwise). The resulting \abbr{psrw} $\{K_t\}$ has the same law as 
the \abbr{srw} $\{B_t\}$ on the fixed graph $\L$ which is transient
(having finite effective resistance between $0$ and $\infty$). 
The number of steps $\tau_i$ it takes $\{B_t\}$ to travel from any 
junction site in $\J$ to one of its neighboring junction 
sites, are i.i.d. random variables whose mean being precisely 
the number of steps it takes the \abbr{srw} on $\Z$ to reach 
from $0$ to $\pm L$. Thus, $\E \tau_1$ is of
order $L^2$ (by diffusivity of the \abbr{srw} on $\Z$). 
Further, during such $\tau_i$ steps at most $2dL$ vertices of 
$\L$ are visited by our \abbr{psrw}, hence at most 
$(2d)^2L$ new probes are being used. Denoting by
$n(t)$ the number of visits made by $\{K_s, s \le t\}$
to the subset $\J$, recall that a.s.
$t^{-1} n(t) \to 1/\E \tau_1$
%by the renewal \abbr{lln}
hence  
\begin{align*}
\overline{m}_t \le \; t^{-1} (2d)^2 L (n(t)+1) \,,
\end{align*}
is eventually bounded above by $c_d/L$
for some non-random $c_d$ finite and all $L$
(which for $L \to \infty$ is made as small as one wishes).
In case $d=2$, we modify the preceding construction by using our 
probes upon first visit of the \abbr{psrw} to sites $z \in \J$  
{\emph {only}} for opening
its right/up/down adjacent edges in $\L$.
The projection of the resulting \abbr{psrw} to the subset $\J$ of the 
stretched lattice $\L$ is then a lazy version of the 
\abbr{fobt} interaction model considered in Proposition \ref{biased_opening} (for $\bD_0=\{0\}$ and
having here probability $1-1/L$ of returing to the
current position before reaching an adjacent junction site). 
Since by Proposition \ref{biased_opening} the sample path of 
this \abbr{fobt} walk on $\J$ is a.s. transient, the same 
applies for our \abbr{psrw}, while by the preceding reasoning  
$\overline{m}_t$ is made as small as one wishes upon
choosing $L \to \infty$.

As for the stated sample path recurrence, fix $d,M \ge 1$ and the 
one-dimensional subspace $\mathbb{O}:=\bZ\times\{0\}\times...\times\{0\}$ of $\bZ^d$. Here we use $2M$ probes each time step, 
placing these on the edges within $\mathbb{O}$ that are 
adjacent to the currently symmetric open interval 
$\bD_t \subset \mathbb{O}$. The resulting \abbr{psrw} 
has $\overline{m}_t=2M$ as large as we wish and 
merely follows the path of the recurrent one-dimensional 
\abbr{srw} on $\mathbb{O}$. 

\noindent
(b). Here a probe emitted at time $t$ follows 
a \abbr{srw} on $\bD_t$, starting from the current 
position $K_t$ of our \abbr{psrw}. Fixing 
$d \ge 2$, we now opt to release at 
the first visit of our \abbr{psrw} to each site 
$z \in\mathbb{Z}^d$, the $\mathcal{F}_t$-adapted
minimal number of probes $m(t)$ required for 
opening all $2d$ edges of $\Z^d$ adjacent to $z$
(i.e. $m(t)=\inf\{s \ge 1: \B^{\Z^d}(z,1) \subseteq \D_{t+s}\}$). It results with the sample path of 
$\{K_t\}$ matching that of a \abbr{srw} on
$\Z^d$, which is thereby a.s. transient when $d \ge 3$
and a.s. recurrent when $d=2$.
Further, the sequence $\{m(t)\}$ of probe counts is 
stochastically dominated by the i.i.d. 
variables $\xi_t-1$, with $\xi_1$ following
the $2d$ coupon collector distribution (i.e. 
the number of independent, uniform samples 
from among $2d$ distinct coupons one needs
for possessing a complete set). Thus, by 
the \abbr{slln}, almost surely,
\begin{align*}
\limsup_{t\rightarrow\infty}
\overline{m}_t \le \E \xi_1 - 1 = c_d 
\end{align*}
(for $c_d :=2d \sum_{\ell=1}^{2d-1} \ell^{-1}$). 
The same transience/recurrence holds 
even when extra $M$ probes are emitted 
at {\emph{each}} step of the \abbr{psrw} 
(yielding $\overline{m}_t\ge M$ 
arbitrarily large). 
\qed

\vskip 5pt

\noindent
{\bf Acknowledgment} 
We thank G. Kozma and J. Ding for the insight which led 
us to part (a) of Proposition \ref{general_soldier}. We 
further thank G. Kozma for sharing with us his 
unpublished preprints \cite{ABGK,K2}, and Stanford's
Mathematics Research Center for the financial support 
of visit by V.S., during which part of this work was done.

\end{document}